\documentclass[11pt]{article}
\usepackage{etex}
\usepackage[all]{xy}
\usepackage{cite}
\usepackage{amsfonts}
\usepackage{amsthm}
\usepackage{enumerate}
\usepackage{graphicx}
\usepackage{mathrsfs}
\usepackage{bm}
\usepackage{amssymb,amsmath} 
\usepackage{makecell}
\usepackage{tikz}
\usepackage{pgf}
\usepackage{tikz}
\usetikzlibrary{patterns}
\usepackage{pgffor}
\usepackage{pgfcalendar}
\usepackage{pgfpages}
\usepackage{shuffle,yfonts}
\usepackage{mathtools}
\DeclareFontFamily{U}{shuffle}{}
\DeclareFontShape{U}{shuffle}{m}{n}{ <-8>shuffle7 <8->shuffle10}{}

\newcommand{\nc}{\newcommand}
\nc{\AMZV}{\mathsf {AMZV}}
\nc{\ud}{\mathrm{d}}
\nc{\ES}{\mathsf {ES}}
\nc{\MZV}{\mathsf {MZV}}
\nc{\MtV}{\mathsf {MtV}}
\nc{\MTV}{\mathsf {MTV}}
\nc{\MSV}{\mathsf {MSV}}
\nc{\MMV}{\mathsf {MMV}}
\nc{\MMVo}{\mathsf {MMVo}}
\nc{\MMVe}{\mathsf {MMVe}}
\nc{\AMMV}{\mathsf {AMMV}}
\nc{\AMTV}{\mathsf {AMTV}}
\nc{\AMtV}{\mathsf {AMtV}}
\nc{\AMSV}{\mathsf {AMSV}}
\nc{\CMZV}{\mathsf {CMZV}}
\nc{\sha}{\shuffle}
\nc{\cst}{\rotatebox[origin=c]{180}{$\sha$}}
\nc{\cstt}{\rotatebox[origin=c]{180}{$\scriptstyle \sha$}}
\nc{\de}{\delta}
\nc{\DD}{{\mathbb D}}
\nc{\anbb}[1]{\left\langle#1\right\rangle}
\nc{\bibb}[1]{\left\{#1\right\}}
\nc{\mibb}[1]{\left[#1\right]}
\nc{\smbb}[1]{\left(#1\right)}
\nc{\doubb}[1]{\llbracket#1\rrbracket}
\nc{\dm}[1]{\left|#1\right|}

\nc{\Gbinom}[2]{\genfrac{(}{)}{0mm}{0}{#1}{#2}}
\nc{\gbinom}[2]{\genfrac{(}{)}{0mm}{1}{#1}{#2}}
\nc{\Rbinom}[2]{\genfrac{\langle}{\rangle}{0mm}{0}{#1}{#2}}
\nc{\rbinom}[2]{\genfrac{\langle}{\rangle}{0mm}{1}{#1}{#2}}
\nc{\Qbinom}[2]{\genfrac{[}{]}{0mm}{0}{#1}{#2}_q}
\nc{\qbinom}[2]{\genfrac{[}{]}{0mm}{1}{#1}{#2}_q}

\nc{\binq}[2]{\genfrac{[}{]}{0mm}{0}{#1}{#2}}
\nc{\tbnq}[2]{\genfrac{[}{]}{0mm}{1}{#1}{#2}}
\nc{\cinq}[2]{\genfrac{\{}{\}}{0mm}{0}{#1}{#2}}
\nc{\tcnq}[2]{\genfrac{\{}{\}}{0mm}{1}{#1}{#2}}

\nc{\mfrac}[2]{\genfrac{}{}{0pt}{}{#1}{#2}}
\nc{\tf}{\tfrac}
\nc{\db}{{\mathbb D}}
\nc{\pari}{{\rm par}}
\nc{\dk}{{\mathbb K}}
\nc{\ola}{\overleftarrow}
\nc{\ora}{\overrightarrow}
\nc{\lra}{\longrightarrow}
\nc{\Lra}{\Longrightarrow}
\nc\Res{{\rm Res}}
\nc\setX{{\mathsf{X}}}
\nc\fA{{\mathfrak{A}}}
\nc\evaM{{\texttt{M}}}
\nc\evaML{{\text{\em{\texttt{M}}}}}
\nc\z{{\texttt{z}}}
\nc\emz{\emph{\texttt{z}}}
\nc\tx{{\texttt{x}}}
\nc\txp{{\tx_1}} 
\nc\txn{{\tx_{-1}}} 
\nc\neo{{1}}
\nc{\yi}{{1}}
\nc\one{{-1}}
\nc\gD{{\Delta}}
\nc\eps{{\varepsilon}}
\nc{\bfMB}{{\bf MB}}
\nc{\bftB}{{\bf tB}}
\nc{\bfTB}{{\bf TB}}
\nc{\bfSB}{{\bf SB}}
\nc{\bfB}{{\bf B}}
\nc{\bfp}{{\bf p}}
\nc{\bfq}{{\bf q}}
\nc{\bfr}{{\bf r}}
\nc{\bfu}{{\bf u}}
\nc{\bfv}{{\bf v}}
\nc{\bfw}{{\bf w}}
\nc{\bfy}{{\bf y}}
\nc{\T}{\ddot{t}}
\nc{\bfe}{{\boldsymbol{\sl{e}}}}
\nc{\bfi}{{\boldsymbol{\sl{i}}}}
\nc{\bfj}{{\boldsymbol{\sl{j}}}}
\nc{\bfk}{{\boldsymbol{\sl{k}}}}
\nc{\bfl}{{\boldsymbol{\sl{l}}}}
\nc{\bfm}{{\boldsymbol{\sl{m}}}}
\nc{\bfn}{{\boldsymbol{\sl{n}}}}
\nc{\bfs}{{\boldsymbol{\sl{s}}}}
\nc{\bft}{{\boldsymbol{\sl{t}}}}
\nc{\bfx}{{\boldsymbol{\sl{x}}}}
\nc{\bfz}{{\boldsymbol{\sl{z}}}}
\nc\bfgs{{\boldsymbol \gs}}
\nc\bfgl{{\boldsymbol \lambda}}
\nc\bfsi{{\boldsymbol \gs}}
\nc\bfet{{\boldsymbol \eta}}
\nc\bfeta{{\boldsymbol \eta}}
\nc\bfeps{{\boldsymbol \eps}}
\nc\mmu{{\boldsymbol \mu}}
\nc\bfone{{\bf 1}}
\nc{\myone}{{1}}

 \nc{\calA}{{\mathcal A}}
 \nc{\calB}{{\mathcal B}}
 \nc{\calC}{{\mathcal C}}
 \nc{\calD}{{\mathcal D}}
 \nc{\calE}{{\mathcal E}}
 \nc{\calF}{{\mathcal F}}
 \nc{\calG}{{\mathcal G}}
 \nc{\calH}{{\mathcal H}}
 \nc{\calI}{{\mathcal I}}
 \nc{\calJ}{{\mathcal J}}
 \nc{\calK}{{\mathcal K}}
 \nc{\calL}{{\mathcal L}}
 \nc{\calM}{{\mathcal M}}
 \nc{\calN}{{\mathcal N}}
 \nc{\calO}{{\mathcal O}}
 \nc{\calP}{{\mathcal P}}
 \nc{\calQ}{{\mathcal Q}}
 \nc{\calR}{{\mathcal R}}
 \nc{\calS}{{\mathcal S}}
 \nc{\calT}{{\mathcal T}}
 \nc{\calU}{{\mathcal U}}
 \nc{\calV}{{\mathcal V}}
 \nc{\calW}{{\mathcal W}}
 \nc{\calX}{{\mathcal X}}
 \nc{\calY}{{\mathcal Y}}
 \nc{\calZ}{{\mathcal Z}}
  \nc{\cala}{{\mathcal a}}
 \nc{\calb}{{\mathcal b}}
 \nc{\calc}{{\mathcal c}}
 \nc{\cald}{{\mathcal d}}
 \nc{\cale}{{\mathcal e}}
 \nc{\calf}{{\mathcal f}}
 \nc{\calg}{{\mathcal g}}
 \nc{\calh}{{\mathcal h}}
 \nc{\cali}{{\mathcal i}}
 \nc{\calj}{{\mathcal j}}
 \nc{\calk}{{\mathcal k}}
 \nc{\call}{{\mathcal l}}
 \nc{\calm}{{\mathcal m}}
 \nc{\caln}{{\mathcal n}}
 \nc{\calo}{{\mathcal o}}
 \nc{\calp}{{\mathsf p}}
 \nc{\calq}{{\mathcal q}}
 \nc{\calr}{{\mathcal r}}
 \nc{\cals}{{\mathcal s}}
 \nc{\calt}{{\mathcal t}}
 \nc{\calu}{{\mathcal u}}
 \nc{\calv}{{\mathcal v}}
 \nc{\calw}{{\mathcal w}}
 \nc{\calx}{{\mathcal x}}
 \nc{\caly}{{\mathcal y}}
 \nc{\calz}{{\mathcal z}}
 \nc{\ot}{{\otimes}}
\usetikzlibrary{arrows,shapes,chains}
 \allowdisplaybreaks

\catcode`!=11
\let\!int\int \def\int{\displaystyle\!int}
\let\!lim\lim \def\lim{\displaystyle\!lim}
\let\!sum\sum \def\sum{\displaystyle\!sum}
\let\!sup\sup \def\sup{\displaystyle\!sup}
\let\!inf\inf \def\inf{\displaystyle\!inf}
\let\!cap\cap \def\cap{\displaystyle\!cap}
\let\!max\max \def\max{\displaystyle\!max}
\let\!min\min \def\min{\displaystyle\!min}
\let\!frac\frac \def\frac{\displaystyle\!frac}
\catcode`!=12

\nc{\gam}{{\gamma}}
\nc{\gG}{{\Gamma}}
\nc{\om}{{\omega}}
\nc{\vep}{{\varepsilon}}
\nc{\ga}{{\alpha}}
\nc{\gl}{{\lambda}}
\nc{\gb}{{\beta}}
\nc{\gd}{{\delta}}
\nc{\gf}{{\varphi}}
\nc{\gs}{{\sigma}}
\nc{\gk}{{\kappa}}
\nc{\gS}{\Sigma}
\let\oldsection\section
\renewcommand\section{\setcounter{equation}{0}\oldsection}

\allowdisplaybreaks

\DeclareMathOperator*{\dep}{dep}
\DeclareMathOperator{\Li}{Li}

\nc\UU{\mbox{\bfseries U}}
\nc\FF{\mbox{\bfseries \itshape F}}
\nc\h{\mbox{\bfseries \itshape h}}\nc\dd{\mbox{d}}
\nc\g{\mbox{\bfseries \itshape g}}
\nc\xx{\mbox{\bfseries \itshape x}}
\def\R{\mathbb{R}}

\def\N{\mathbb{N}}
\def\Z{\mathbb{Z}}
\def\Q{\mathbb{Q}}

\def\xx{\left(\frac{1-x}{1+x} \right)}

\nc\divg{{\text{div}}}
\theoremstyle{plain}
\newtheorem{thm}{Theorem}[section]

\newtheorem{cor}[thm]{Corollary}
\newtheorem{con}[thm]{Conjecture}
\newtheorem{pro}[thm]{Proposition}
\theoremstyle{definition}

\nc{\cicc}[1]{{}_{{}^{ \bigcirc\hskip-1.2ex{#1}\hskip.3ex{}}}}
\nc{\cic}[1]{{}^{\bigcirc\hskip-1.15ex{\raisebox{-0.015cm}{\text{$\scriptscriptstyle #1$}}}\hskip.25ex{}}}
\nc{\ccic}[1]{{}^{\bigcirc\hskip-1.5ex{\raisebox{-0.015cm}{\text{$\scriptscriptstyle #1$}}}\hskip.25ex{}}}
\nc{\ncic}[1]{ {\bigcirc\hskip-1.6ex{\raisebox{-0.0cm}{\text{$\scriptstyle #1$}}}\hskip.25ex{}}}
\nc{\nncic}[1]{ {\bigcirc\hskip-2ex{\raisebox{-0.0cm}{\text{$\scriptstyle #1$}}}\hskip.25ex{}}}
\nc{\cci}[1]{{}_{{}^{ {\textstyle \bigcirc}\hskip-2.05ex{#1}\hskip-.35ex{}}}}
\nc{\ccicc}[1]{{}_{{}^{ {\textstyle \bigcirc}\hskip-1.55ex{#1}\hskip-0.1ex{}}}}
\nc{\x}{\rm{x}}
\nc{\tworow}[2]{\left(#1 \atop #2\right)}

\nc{\fl}{{\mathfrak l}}
\nc{\fm}{{\mathfrak m}}

\begin{document}
\title{\bf On the Proof of the Gen\v{c}ev-Rucki Conjecture for Multiple Ap\'ery-Like Series}
\author{
{Ce Xu\thanks{Email: cexu2020@ahnu.edu.cn}}\\[1mm]
\small School of Mathematics and Statistics, Anhui Normal University,\\ \small Wuhu 241002, P.R. China
}
\date{}
\maketitle

\noindent{\bf Abstract.}
In this paper, we employ the theories and techniques of hypergeometric functions to provide two distinct proofs of the conjectured identities involving multiple Ap\'ery-like series with central binomial coefficients and multiple harmonic star sums, as recently proposed by Gen\v{c}ev and Rucki. Furthermore, we establish several more general identities for multiple Ap\'ery-like series. Furthermore, by utilizing the method of iterated integrals, a class of multiple mixed values can be expressed as combinations of the multiple Ap\'ery-like series identities conjectured by Gen\v{c}ev and Rucki and $\zeta(2,\ldots,2)$, thus allowing explicit formulas for these multiple mixed values to be derived in terms of Riemann zeta values.

\medskip

\medskip


\section{Introduction}
For integers $k_1,\ldots,k_r\geq 1$, a finite sequence $\bfk:=(k_1,\ldots, k_r)\in\N^r$ is called a \emph{composition}. For a composition $\bfk=(k_1,\ldots,k_r)$, we put
\begin{equation*}
 |\bfk|:=k_1+\cdots+k_r,\quad \dep(\bfk):=r,
\end{equation*}
and call them the weight and the depth of $\bfk$, respectively. If $k_1>1$, $\bfk$ is called \emph{admissible}.
As a convention, we denote by $\{m\}_r$ the sequence of $m$'s with $r$ repetitions.

For a composition $\bfk=(k_1,\ldots,k_r)$ and positive integer $n$, the \emph{multiple harmonic sums} and \emph{multiple harmonic star sums} are defined by
\begin{align}
\zeta_n(\bfk):=\sum\limits_{n\geq n_1>\cdots>n_r>0 } \frac{1}{n_1^{k_1}\cdots n_r^{k_r}}\in \Q\quad
\text{and}\quad
\zeta^\star_n(\bfk):=\sum\limits_{n\geq n_1\geq\cdots\geq n_r>0} \frac{1}{n_1^{k_1}\cdots n_r^{k_r}}\in \Q\label{MHSs+MHSSs},
\end{align}
respectively. If $n<r$ then ${\zeta_n}(\bfk):=0$ and ${\zeta _n}(\emptyset )={\zeta^\star _n}(\emptyset ):=1$. When $\bfk$ is admissible, by taking the limit $n\rightarrow \infty$ in \eqref{MHSs+MHSSs} we get the \emph{multiple zeta values} (MZVs) and the \emph{multiple zeta star values}, respectively
\begin{align*}
{\zeta}( \bfk):=\lim_{n\rightarrow \infty}{\zeta _n}(\bfk)\in \R \quad
\text{and}\quad
{\zeta^\star}( \bfk):=\lim_{n\rightarrow \infty}{\zeta^\star_n}( \bfk)\in \R.
\end{align*}
When the depth is 1, we recover the classical Riemann zeta values $\zeta(n)$. The concept of multiple zeta values was independently introduced in the early 1990s by Hoffman \cite{H1992} and Zagier \cite{DZ1994}. Owing to their profound connections with diverse mathematical and physical fields-including knot theory, algebraic geometry, and theoretical physics-the study of multiple zeta values has attracted sustained and widespread interest among researchers. After more than three decades of development, the subject has generated a substantial body of research. For an authoritative synthesis of results up to 2016, we refer the reader to Zhao's comprehensive monograph \cite{Z2016}.

The most fundamental problem concerning multiple zeta values is the study of their relations over the rational number field $\mathbb{Q}$. Euler's celebrated result $\zeta(2) = \pi^2/6$, together with the transcendence of $\pi$, implies the transcendence of $\zeta(2)$. His approach can be naturally extended to show that every Riemann zeta value at even positive integers $\zeta(2n)$ is a rational multiple of $\pi^{2n}$. In contrast, although it is widely conjectured that $\zeta(2n+1)$ is transcendental for all $n \geq 1$, not a single instance of this has been proven to date. A major breakthrough came in 1978, when Ap\'ery \cite{Apery1978} astonished the mathematical community by proving the irrationality of $\zeta(3)$. His proof relied on a clever and intricate analysis of the series
\begin{equation}\label{equ:AperySeries}
\zeta(3) = \frac{5}{2} \sum_{n =1}^\infty \frac{(-1)^{n-1}}{n^3 \binom{2n}{n}}.
\end{equation}
Due to Ap\'ery's contributions, series whose general term contains central binomial coefficients in either the numerator or denominator are generally referred to as \emph{Ap\'ery-like series}. Furthermore, if the general term of an Ap\'ery-like series contains multiple harmonic (star) sums in the numerator, we refer to it as a \emph{multiple Ap\'ery-like series}. The study of (multiple) Ap\'ery-like series has attracted considerable attention from numerous experts and scholars. For some recent work on the subject, we refer the reader to \cite{Au2024,CantariniD2019,Chen2016,ChenWangZhong2025,ChenWang2025,GR2025,LaiLuorr2022,LiYan2025,Lupu2022,XuZhao2021c} and the references therein. Recently, while studying explicit formulas for multiple Ap\'ery-like series of the form
\begin{align*}
\sum_{n_1\geq \cdots\geq n_r\geq 1} \frac{\binom{2n_1}{n_1}}{n_14^{n_1}} \frac{4^{n_r}}{\binom{2n_r}{n_r}}\prod\limits_{j=2}^r \frac{1}{n_j^2}
\end{align*}
where both the largest summation index $n_1$ and and the smallest index $n_r$ involve central binomial coefficients, Gen\v{c}ev and Rucki proposed the following conjecture:
\begin{con}\label{con-one} (Gen\v{c}ev and Rucki, Conjecture, \cite{GR2025})
For an arbitrary $r\in\N$ with $r>1$, we have
\begin{align}\label{equ-con-one}
\sum_{n=1}^\infty \frac{\binom{2n}{n}}{n4^n}\zeta_n^\star(\{2\}_r)=2(1-4^{-r})\zeta(2r+1).
\end{align}
\end{con}
In this paper, we employ two distinct hypergeometric function methods to prove this conjecture and further establish explicit evaluations for several other related series.

\section{Proof of Conjecture \ref{con-one} via Hypergeometric Series}

In this section, we employ transformation formulas for hypergeometric series to prove Conjecture \ref{con-one}, and further establish generalizations of the identity in Conjecture 1.1 to its functional form. We first need to express the series on the left-hand side of \eqref{equ-con-one} in the form of a hypergeometric series.

We note that
\begin{align}\label{equ-b-one}
\sum_{r=0}^\infty \zeta_n^\star(\{2\}_r)x^{2r} =\prod_{j=1}^n \left(1+\frac{x^2}{j^2}+\frac{x^4}{j^4}+\cdots\right)=\prod_{j=1}^n \left(1-\frac{x^2}{j^2}\right)^{-1}=\frac{(1)_n^2}{(1-x)_n(1+x)_n},
\end{align}
where $(x)_n:=x(x+1)\cdots(x+n-1)$ and $(x)_0:=1$. Multiplying \eqref{equ-b-one} by $\frac{\binom{2n}{n}}{n 4^n}$ and summing $n$ from $1$ to $\infty$ yields
\begin{align}\label{equ-b-two}
\sum_{r=0}^\infty \left\{\sum_{n=1}^\infty \frac{\binom{2n}{n}}{n 4^n} \zeta_n^\star(\{2\}_r)\right\}x^{2r} =\sum_{n=1}^\infty \frac{(1/2)_n (1)_n}{(1-x)_n(1+x)_n} \frac{1}{n}=\frac{1}{2(1-x^2)}{}_3F_2\left({1,1,3/2\atop 2-x,2+x};1\right),
\end{align}
where the \emph{general hypergeometric series} ${}_pF_q$ is defined as
\begin{align*}
{}_pF_q\left({a_1,a_2,\ldots,a_p\atop b_1,b_2,\ldots,b_q};z\right):=\sum_{n=0}^\infty \frac{(a_1)_n(a_2)_n\cdots (a_p)_n}{(b_1)_n(b_2)_n\cdots(b_q)_n}\frac{z^n}{n!}.
\end{align*}

By applying transformation formulas for hypergeometric series, we can derive a closed-form evaluation of the hypergeometric series on the right-hand side of \eqref{equ-b-two}.
\begin{thm}\label{thm-one-hydigama} For $x\in \mathbb{C}\setminus \Z_0\ (\Z_0:=\Z\setminus \{0\})$, we have
\begin{align}\label{thm-one-equ-hydigama}
{}_3F_2\left({1,1,3/2\atop 2-x,2+x};1\right)&=4\log(2)(1-x^2)+2(1-x^2)(\psi(1-x/2)+\psi(1+x/2))\nonumber\\&\quad-2(1-x^2)(\psi(1-x)+\psi(1+x)),
\end{align}
where $\psi(s)$ denotes the digamma function defined by
\begin{align}\label{defn-classical-psi-funtion}
\psi(s)=-\gamma-\frac1{s}+\sum_{k=1}^\infty \left(\frac1{k}-\frac{1}{s+k}\right),
\end{align}
where $s\in\mathbb{C}\setminus \N_0^-$ and $\N_0^-:=\N^-\cup\{0\}=\{0,-1,-2,-3,\ldots\}$. Here $\gamma$ denotes the \emph{Euler-Mascheroni constant}.
\end{thm}
\begin{proof}
From \cite[Eq. (3.1.15)]{A2000}, we have
\begin{align}\label{ref-cite-hygem-funct-1}
{}_3F_2\left({a,b,c\atop a-b+1,a-c+1};z \right)=(1-z)^{-a}{}_3F_2\left({a-b-c+1,\frac{a}{2},\frac{a+1}{2}\atop a-b+1,a-c+1};-\frac{4z}{(1-z)^2}\right).
\end{align}
Setting $a=2,b=1+x,c=1-x$ and $z=-1$ gives
\begin{align}\label{equ-hyperfunct-digamam}
{}_3F_2\left({1,1,3/2\atop 2-x,2+x};1\right)&=4\ {}_3F_2\left({2,1+x,1-x\atop 2-x,2+x};-1\right)\nonumber\\
&=4\sum_{n=0}^\infty \frac{(2)_n(1+x)_n(1-x)_n}{(2-x)_n(2+x)_n} \frac{(-1)^n}{n!}\nonumber\\
&=4(1-x^2)\sum_{n=0}^\infty \frac{(n+1)(-1)^n}{(n+1-x)(n+1+x)}\nonumber\\
&=2(1-x^2)\sum_{n=0}^\infty \left(\frac{(-1)^n}{n+1-x}+\frac{(-1)^n}{n+1+x} \right)\nonumber\\
&=(1-x^2)\left(\psi\Big(1-\frac{x}{2}\Big)-\psi\Big(\frac{1-x}{2}\Big)+\psi\Big(1+\frac{x}{2}\Big) -\psi\Big(\frac{1+x}{2}\Big)\right).
\end{align}
By \emph{multiplication theorem} (see \cite[Page 913, 8.365]{GR})
\begin{align}
\psi(nz)=\frac1{n}\sum_{k=0}^{n-1}\psi\Big(z+\frac{k}{n}\Big)+\log 2,
\end{align}
letting $n=2$ yields
\begin{align}
\psi(z)=2\psi(2z)-\psi\Big(z+\frac1{2}\Big)-2\log2.
\end{align}
Hence, letting $z=1\pm x$ and add them together, then
\begin{align}
\psi\Big(\frac{1+x}{2}\Big)+\psi\Big(\frac{1-x}{2}\Big)=2\Big(\psi(1+x)+\psi(1-x)\Big)-\left( \psi\Big(1+\frac{x}{2}\Big)+\psi\Big(1-\frac{x}{2}\Big)\right)-4\log2.
\end{align}
Substituting it into \eqref{equ-hyperfunct-digamam} yields the desired evaluation.
\end{proof}

We now proceed to prove Conjecture \ref{con-one}.
\begin{thm}\label{thm-Apery-zeta}For $r\in \N$, we have
\begin{align} \label{thm-equ-Apery-zeta}
\sum_{n=1}^\infty \frac{\binom{2n}{n}}{n4^n}\zeta_n^\star(\{2\}_r)=2(1-4^{-r})\zeta(2r+1).
\end{align}
\end{thm}
\begin{proof}
Applying \eqref{thm-one-equ-hydigama} and by utilizing the power series expansion of the digamma function at zero
\[\psi(1+x)+\gamma=\sum_{n=1}^\infty (-1)^{n-1}\zeta(n+1)x^n,\]
we obtain
\begin{align}\label{proof-thm-two-proecess-two}
\frac{1}{2(1-x^2)}{}_3F_2\left({1,1,3/2\atop 2-x,2+x};1\right)=2\log2+2\sum_{n=1}^\infty (1-4^{-n})\zeta(2n+1)x^{2n}.
\end{align}
Finally, by comparing the coefficients of $x^{2r}$ on both sides of \eqref{equ-b-one} and \eqref{proof-thm-two-proecess-two} yields the desired result.
\end{proof}

Indeed, by employing similar methods, we can establish the functional forms of equations \eqref{thm-one-equ-hydigama} and \eqref{thm-equ-Apery-zeta}.
\begin{thm}\label{thm-one-hyparadigama} For $x\in \mathbb{C}\setminus \Z_0$ and $t\in [0,1]$, we have
\begin{align}\label{thm-one-equ-hyparadigama}
{}_3F_2\left({1,1,3/2\atop 2-x,2+x};1-t\right)=\frac{2(1-x^2)}{1-t} \sum_{n=1}^\infty \left(\frac{\sqrt{t}-1}{\sqrt{t}+1}\right)^n \left(\frac{1}{x-n}-\frac1{x+n}\right).
\end{align}
\end{thm}
\begin{proof}
The theorem follows immediately from \eqref{ref-cite-hygem-funct-1} if we set $a=2,b=1+x,c=1-x$ and $z=\frac{\sqrt{t}-1}{\sqrt{t}+1}$.
\end{proof}

\begin{thm}\label{thm-Apery-zeta-para-MPL}For $r\in \N$ and $z\in [0,1]$, we have
\begin{align} \label{thm-equ-Apery-zeta-para-MPL}
\sum_{n=1}^\infty \frac{\binom{2n}{n}}{n4^n}\zeta_n^\star(\{2\}_r)(1-z)^n=-2\Li_{2r+1}\left(\frac{\sqrt{z}-1}{\sqrt{z}+1}\right),
\end{align}
where the \emph{polylogarithm function} $\Li_p(x)$ is defined by
\begin{align}
\Li_{p}(x):=\sum_{n=1}^\infty \frac{x^n}{n^p}\quad (x\in[-1,1],\ (p,x)\neq (1,1),\ p\in \N).
\end{align}
\end{thm}
\begin{proof}
By an elementary calculation, we find that
\begin{align}\label{equ-b-two-para}
\sum_{r=0}^\infty \left\{\sum_{n=1}^\infty \frac{\binom{2n}{n}}{n 4^n} \zeta_n^\star(\{2\}_r)(1-z)^n\right\}x^{2r}& =\sum_{n=1}^\infty \frac{(1/2)_n (1)_n}{(1-x)_n(1+x)_n} \frac{(1-z)^n}{n}\nonumber\\&=\frac{1-z}{2(1-x^2)}{}_3F_2\left({1,1,3/2\atop 2-x,2+x};1-z\right).
\end{align}
Noting that from \eqref{thm-one-equ-hyparadigama} we obtain
\begin{align}\label{equ-b-two-para-2}
\frac{1-z}{2(1-x^2)}{}_3F_2\left({1,1,3/2\atop 2-x,2+x};1-z\right)=-2\sum_{r=0}^\infty \Li_{2r+1} \left(\frac{\sqrt{z}-1}{\sqrt{z}+1}\right)x^{2r}.
\end{align}
Hence, comparing the the coefficients of $x^{2r}$ on both sides of above yields the desired evaluation.
\end{proof}
Clearly, Conjecture \ref{con-one} is obtained by setting $z=0$ in Theorem \ref{thm-Apery-zeta-para-MPL}.

\section{Closed Forms of A Family of Multiple Mixed Values}
In this section, we employ the method of iterated integrals, combined with the results from the previous section, to derive closed-form formulas expressing a class of multiple mixed values involving alternating harmonic numbers and multiple harmonic sums in terms of Riemann zeta values.
\begin{thm}\label{thm-MZVs-zeta} For $r\in \N$,
\begin{align}\label{thm-equ-MZVs-zeta}
\sum_{n=1}^\infty \frac{{\bar H}_{2n}\zeta_{n-1}(\{2\}_{r-1})}{n^2}=\sum_{j=1}^{r}(-1)^{j}(1-4^{-j})\zeta(\{2\}_{r-j})\zeta(2j+1),
\end{align}
where the alternating harmonic number ${\bar H}_{2n}$ is defined by ${\bar H}_{2n}:=\sum_{j=1}^{2n} \frac{(-1)^j}{j}$. From \cite[Cor. 5.6.4]{Z2016}, we have
\begin{align*}
\zeta(\{2\}_r)=\frac{\pi^{2r}}{(2r+1)!}\quad (r\in\N).
\end{align*}
\end{thm}
\begin{proof}
From \cite[Theorem 2.1]{XuZhao2020d}, replacing $p$ by $r$ and setting $x=1$ and all $m_j=2$, we have
\begin{align}\label{equ-one}
n\int_0^1 t^{n-1}dt\left\{\frac{dt}{1-t}\frac{dt}{t}\right\}^r=\sum_{j=1}^{r+1}(-1)^{j+1}\zeta^\star_n(\{2\}_{j-1})\zeta(\{2\}_{r-j+1})
\end{align}
where the iterated integral is defined in the following form
$$
\int_{a}^b f_1(t)dtf_2(t)dt\cdots f_p(t)dt:=\int\limits_{a<t_1<\cdots<t_p<b} f_p(t_p)f_{p-1}(t_{p-1})\cdots f_1(t_1)dt_1dt_2\cdots dt_p.
$$
Multiplying \eqref{equ-one} by $\frac{\binom{2n}{n}}{n 4^n}$ and summing $n$ from $1$ to $\infty$ and noting the fact that
\[\sum_{n=0}^\infty \frac{\binom{2n}{n}}{4^n}t^n=\frac1{\sqrt{1-t}}\quad (t\in [-1,1)),\]
we obtain
\begin{align}\label{equ-two}
\sum_{j=1}^{r+1}(-1)^{j+1}\zeta(\{2\}_{r-j+1})\sum_{n=1}^\infty \frac{\binom{2n}{n}}{n 4^n}\zeta^\star_n(\{2\}_{j-1})=\int_0^1 \left\{\frac1{t}\Big(\frac1{\sqrt{1-t}}-1\Big)dt\right\}\left\{\frac{dt}{1-t}\frac{dt}{t}\right\}^r.
\end{align}
Applying $t\mapsto 1-t$ yields
\begin{align}\label{equ-three}
\sum_{j=1}^{r+1}(-1)^{j+1}\zeta(\{2\}_{r-j+1})\sum_{n=1}^\infty \frac{\binom{2n}{n}}{n 4^n}\zeta^\star_n(\{2\}_{j-1})=\int_0^1 \left\{\frac{dt}{1-t}\frac{dt}{t}\right\}^r\left\{\frac1{1-t}\Big(\frac1{\sqrt{t}}-1\Big)dt\right\}.
\end{align}
Then letting $t\mapsto t^2$ gives
\begin{align}\label{equ-four}
\sum_{j=1}^{r+1}(-1)^{j+1}\zeta(\{2\}_{r-j+1})\sum_{n=1}^\infty \frac{\binom{2n}{n}}{n 4^n}\zeta^\star_n(\{2\}_{j-1})&=\int_0^1 \left\{\frac{2t dt}{1-t^2}\frac{2tdt}{t^2}\right\}^r\left\{\frac1{1-t^2}\Big(\frac1{t}-1\Big)2tdt\right\}\nonumber\\
&=\int_0^1 \left\{\frac{2t dt}{1-t^2}\frac{2dt}{t}\right\}^r\frac{2dt}{1+t}\nonumber\\
&=2\int_0^1 \frac{\Li_{\{2\}_r}(t^2)}{1+t}dt,
\end{align}
where the classical \emph{multiple polylogarithm} (MPL) is defined by
\begin{align}
\Li_{k_1,\ldots,k_r}(x)&:=\sum_{n_1>n_2>\cdots>n_r>0} \frac{x_1^{n_1}}{n_1^{k_1}\dotsm n_r^{k_r}}\quad (\forall k_j\in\N,\ (k_1,x)\neq (1,1))\\
&=\int_0^x \left(\frac{dt}{1-t}\right)\left(\frac{dt}{t}\right)^{k_r-1}\cdots
\left(\frac{dt}{1-t}\right)\left(\frac{dt}{t}\right)^{k_1-1}\nonumber.
\end{align}

Noting the fact that (see \cite{Chen2016})
\begin{align}
\sum_{n=1}^\infty \frac{\binom{2n}{n}}{n4^n}t^n=2\log\Big(\frac2{1+\sqrt{1-t}}\Big)\quad (t\in [-1,1)),
\end{align}
and letting $t=1$ yields
\begin{align}\label{equ-integral-log}
\sum_{n=1}^\infty \frac{\binom{2n}{n}}{n4^n}=2\log(2).
\end{align}
Substituting it into \eqref{equ-four} gives
\begin{align}\label{equ-five}
&\sum_{j=1}^{r}(-1)^{j}\zeta(\{2\}_{r-j})\sum_{n=1}^\infty \frac{\binom{2n}{n}}{n 4^n}\zeta^\star_n(\{2\}_{j})=2\int_0^1 \frac{\Li_{\{2\}_r}(t^2)}{1+t}dt-2\zeta(\{2\}_r)\log 2\nonumber\\
&=2\int_0^1 \frac{\Li_{\{2\}_r}(t^2)-\zeta(\{2\}_r)}{1+t}dt=2\sum_{n_1>n_2>\cdots>n_r>0}\frac{1}{n_1^2\cdots n_r^2} \int_0^1 \frac{t^{2n_1}-1}{1+t}dt\nonumber\\
&=2\sum_{n_1>n_2>\cdots>n_r>0}\frac{1}{n_1^2\cdots n_r^2} \sum_{j=1}^{n_1}\int_0^1 t^{2j-2}(t-1)dt\nonumber\\
&=2\sum_{n_1>n_2>\cdots>n_r>0}\frac{1}{n_1^2\cdots n_r^2} \sum_{j=1}^{n_1}\Big(\frac{1}{2j}-\frac1{2j-1}\Big)\nonumber\\
&=2\sum_{n_1>n_2>\cdots>n_r>0}\frac{1}{n_1^2\cdots n_r^2} \sum_{j=1}^{2n_1}\frac{(-1)^j}{j}.
\end{align}
Combining \eqref{equ-five} with \eqref{thm-equ-Apery-zeta} yields \eqref{thm-equ-MZVs-zeta}.
\end{proof}

Note that ${\bar H}_{2n}=\frac{H_n-t_n}{2}$, where $H_n$ and $t_n$ stand for the \emph{classical harmonic number} and \emph{classical odd harmonic number}, respectively, which are defined by
\begin{align*}
H_n:=\sum_{k=1}^n \frac1{k}\quad\text{and}\quad t_n:=\sum_{k=1}^n \frac1{k-1/2}\quad (n\in\N).
\end{align*}
Hence, we obtain
\begin{align*}
\sum_{n=1}^\infty \frac{t_n \zeta_{n-1}(\{2\}_{r-1})}{n^2}=\sum_{n=1}^\infty \frac{H_n \zeta_{n-1}(\{2\}_{r-1})}{n^2}-2\sum_{j=1}^{r}(-1)^{j}(1-4^{-j})\zeta(\{2\}_{r-j})\zeta(2j+1).
\end{align*}
From \cite{Ku2010}, we have
\begin{align}\label{Kuba-AKZV}
\xi(2;\{2\}_r)=\sum\limits_{n=1}^\infty \frac{\zeta_{n-1}(k_2,\ldots,k_r)H_n}{n^{k_1+1}},
\end{align}
where $\xi(2;\{2\}_r)$ is actually a special case of the Arakawa-Kaneko zeta values $\xi(p;k_1,k_2\ldots,k_r)$. The \emph{Arakawa-Kaneko zeta function} was introduced by Arakawa and Kaneko in \cite{AM1999}, defined as
\begin{align}\label{defn-Arakawa-Kaneko-zeta-F}
\xi(s;k_1,k_2\ldots,k_r):=\frac{1}{\Gamma(s)} \int\limits_{0}^\infty \frac{t^{s-1}}{e^t-1}{\rm Li}_{k_1,k_2,\ldots,k_r}(1-e^{-t})dt\quad (\Re(s)>0).
\end{align}
For recent developments on this topic, readers may refer to \cite{LuoSi2023,ZY2023} and the references therein.

\section{Another Proof of Hypergeometric Series ${}_3F_2\left({1,1,3/2\atop 2-x,2+x};1\right)$}
In this section, we present an alternative proof of equation \eqref{thm-one-equ-hydigama} using the method of partial fraction decomposition, and furthermore establish several formulas for parametric multiple Ap\'ery-like series.
\begin{pro} Let $\bar{\Z}:=\Z\backslash \{0,1\}$. For $x\in \mathbb{C}\setminus \Z_0$ and $a,b\in \R\backslash \bar{\Z}$ with $a+b\notin\{3,4,5,\ldots\}$,
\begin{align}\label{Hypergenfunvtion0-main-resutl}
{}_3F_2\left({1,a,b\atop 2-x,2+x};1\right)=(x^2-1)\sum_{k=-\infty,\atop k\neq 0}^\infty \frac{(-1)^k}{x+k} \frac{(a)_{|k|-1}(b)_{|k|-1} {\rm sign}(k)|k| \Gamma(3-a-b)}{\Gamma(2+|k|-a)\Gamma(2+|k|-b)}.
\end{align}
\end{pro}
\begin{proof} From the definition of the hypergeometric series ${}_3F_2$, it is straightforward to obtain
\begin{align*}
{}_3F_2\left({1,a,b\atop 2-x,2+x};1\right)&=\sum_{n=0}^\infty \frac{(1)_n(a)_n(b)_n}{(2+x)_n(2-x)_n}\frac1{n!}\nonumber\\
&=\sum_{n=0}^\infty \frac{(a)_n(b)_n}{(x-1-n)_{2n+3}}(-1)^n x(x^2-1).
\end{align*}
Using partial fraction expansion, there are
\begin{align*}
\frac{x}{(x-1-n)_{2n+3}}=\sum_{k=-1-n}^{n+1}(-1)^{n+k}\frac{k}{(n+k+1)!(n-k+1)!} \frac1{x+k}.
\end{align*}
Hence,
\begin{align*}
&{}_3F_2\left({1,a,b\atop 2-x,2+x};1\right)=(x^2-1)\sum_{n=0}^\infty (a)_n(b)_n\sum_{k=-1-n}^{n+1}(-1)^{k}\frac{k}{(n+k+1)!(n-k+1)!} \frac1{x+k}\\
&=(x^2-1)\sum_{k=-\infty}^\infty \frac{(-1)^k k}{x+k} \sum_{n\geq |k|-1} \frac{(a)_n(b)_n}{(n+k+1)!(n-k+1)!}\quad (n=m+|k|-1)\\
&=(x^2-1)\sum_{k=-\infty,\atop k\neq 0}^\infty \frac{(-1)^k |k| {\rm sign}(k)}{x+k} \sum_{m=0}^\infty \frac{(a)_{m+|k|-1}(b)_{m+|k|-1}}{(m+2|k|)!m!}\\
&=(x^2-1)\sum_{k=-\infty,\atop k\neq 0}^\infty \frac{(-1)^k |k| {\rm sign}(k)}{x+k} \sum_{m=0}^\infty \frac{(a)_{|k|-1}(a+|k|-1)_m(b)_{|k|-1}(b+|k|-1)_m}{(2|k|)!(1+|2k|)_mm!}\\
&=(x^2-1)\sum_{k=-\infty,\atop k\neq 0}^\infty \frac{(-1)^k |k| {\rm sign}(k) (a)_{|k|-1}(b)_{|k|-1} }{(x+k)(2|k|)!}{}_2F_1 \left({a+|k|-1,b+|k|-1\atop 1+2|k|};1\right)\\
&=(x^2-1)\sum_{k=-\infty,\atop k\neq 0}^\infty \frac{(-1)^k |k| {\rm sign}(k) (a)_{|k|-1}(b)_{|k|-1} }{(x+k)(2|k|)!} \frac{\Gamma(1+2|k|)\Gamma(3-a-b)}{\Gamma(2+|k|-a)\Gamma(2+|k|-b)}\\
&=(x^2-1)\sum_{k=-\infty,\atop k\neq 0}^\infty \frac{(-1)^k}{x+k} \frac{(a)_{|k|-1}(b)_{|k|-1} {\rm sign}(k)|k| \Gamma(3-a-b)}{\Gamma(2+|k|-a)\Gamma(2+|k|-b)},
\end{align*}
where we used the \emph{Gauss's summation theorem} (see \cite[Page 66, Thm. 2.2.2]{A2000})
\begin{align*}
{}_2F_1\left({a,b\atop c};1\right)=\frac{\Gamma(c)\Gamma(c-a-b)}{\Gamma(c-a)\Gamma(c-b)}.
\end{align*}
This completes the proof of the proposition.
\end{proof}

If letting $a=1$ in \eqref{Hypergenfunvtion0-main-resutl} then
\begin{align}\label{Hypergenfunvtion0-main-resut2}
{}_3F_2\left({1,1,b\atop 2-x,2+x};1\right)&=(x^2-1) \sum_{k=-\infty,\atop k\neq 0}^\infty \frac{(-1)^k}{x+k}\frac{(b)_{|k|-1}{\rm sign}(k)}{(2-b)_{|k|}}\nonumber\\
&=(x^2-1) \sum_{k=1}^\infty (-1)^k \frac{(b)_{k-1}}{(2-b)_{k}} \left(\frac1{x+k}-\frac1{x-k}\right).
\end{align}
Further, setting $b=\frac1{2},\ 1$ and $\frac3{2}$ in \eqref{Hypergenfunvtion0-main-resut2} give
\begin{align}
{}_3F_2\left({1,1,1/2\atop 2-x,2+x};1\right)&=(x^2-1)\sum_{k=1}^\infty \frac{(-1)^k k}{(k^2-1/4)(k^2-x^2)},\\
{}_3F_2\left({1,1,1\atop 2-x,2+x};1\right)&=(x^2-1)\sum_{k=1}^\infty (-1)^k \left\{\frac1{k(x+k)}-\frac{1}{k(x-k)}\right\}\nonumber\\
&=\frac{x^2-1}{x}\left(\frac1{x}-\frac{\pi}{\sin(\pi x)}\right),\\
{}_3F_2\left({1,1,3/2\atop 2-x,2+x};1\right)&=(x^2-1)2\sum_{k=1}^\infty (-1)^k \left(\frac{1}{x+k}-\frac1{x-k}\right)\nonumber\\
&=4\log(2)(1-x^2)+2(1-x^2)(\psi(1-x/2)+\psi(1+x/2))\nonumber\\&\quad-2(1-x^2)(\psi(1-x)+\psi(1+x)).\label{thm-one-equ-hydigama-prv-2}
\end{align}
The final equation \eqref{thm-one-equ-hydigama-prv-2} is consistent with the previous equation \eqref{thm-one-equ-hydigama}.

Finally, we present several results on parametric multiple Ap\'ery-like series. By a direct calculation, we obtain
\begin{align}
\sum_{r=0}^\infty \left\{\sum_{n=1}^\infty \frac{(a)_n}{n!n}\zeta_n^\star(\{2\}_r) \right\}x^{2r}=\frac{a}{1-x^2}{}_3F_2\left({1,1,1+a\atop 2-x,2+x};1\right).
\end{align}
By applying \eqref{Hypergenfunvtion0-main-resut2} with $b=a+1$, we expand the rational expression involving $x$ into a power series. Then, comparing the coefficients of $x^{2r}$ on both sides yields
\begin{align}\label{Hypergenfunvtion0-main-resut3}
\sum_{n=1}^\infty \frac{(a)_n}{n!n}\zeta_n^\star(\{2\}_r)&=2a\sum_{k=1}^\infty (-1)^{k-1} \frac{(1+a)_{k-1}}{(1-a)_{k}}\frac1{k^{2r+1}}\nonumber\\&=2\sum_{n=1}^\infty (-1)^{n-1} \frac{(a)_{n}}{(1-a)_{n}}\frac1{n^{2r+1}}.
\end{align}
Setting $a=1/2$ in \eqref{Hypergenfunvtion0-main-resut3} gives \eqref{thm-equ-Apery-zeta}. This provides an alternative proof of \eqref{thm-equ-Apery-zeta}.

\begin{thm} For integer $k,r\geq 0$ and $a\notin \Z$,
\begin{align}\label{Hypergenfunvtion0-main-resut-import}
\sum_{n=1}^\infty \frac{(a)_n}{n!n}\zeta_n^\star(\{2\}_r)\zeta_n(\{1\}_k;a)=2\sum_{i+j=k,\atop i,j\geq 0} \sum_{n=1}^\infty  \frac{(-1)^{n-1}}{n^{2r+1}} \frac{(a)_{n}}{(1-a)_{n}} \zeta_n(\{1\}_i;a)\zeta^\star_n(\{1\}_j;1-a),
\end{align}
where the \emph{Hurwitz-type multiple harmonic sums} and \emph{Hurwitz-type multiple harmonic star sums} are defined by
\begin{align*}
\zeta_n(\bfk;\alpha):=\sum\limits_{n\geq n_1>\cdots>n_r>0 } \frac{1}{(n_1+\alpha-1)^{k_1}\cdots (n_r+\alpha-1)^{k_r}}
\end{align*}
and
\begin{align*}
\zeta^\star_n(\bfk;\alpha):=\sum\limits_{n\geq n_1\geq\cdots\geq n_r>0} \frac{1}{(n_1+\alpha-1)^{k_1}\cdots (n_r+\alpha-1)^{k_r}},
\end{align*}
respectively. If $n<r$ then ${\zeta_n}(\bfk;\alpha):=0$ and ${\zeta _n}(\emptyset;\alpha)={\zeta^\star _n}(\emptyset;\alpha):=1$.
\end{thm}
\begin{proof}
From \cite[Eqs. (2.29), (2.34)]{KanekoWangXuZhao2022}, we have
\begin{align*}
\frac{d^k}{da^k} (a)_n=k!(a)_n \zeta_n(\{1\}_k;a)\quad \text{and}\quad \frac{d^k}{da^k} \frac1{(1-a)_n}=\frac{k!}{(1-a)_n}\zeta^\star_n(\{1\}_k;1-a).
\end{align*}
Hence, Differentiating $a$ in \eqref{Hypergenfunvtion0-main-resut3} $k$ times and applying the above relation yields the desired result.
\end{proof}

In particular, setting $a=1/2$ in \eqref{Hypergenfunvtion0-main-resut-import} gives the following corollary.
\begin{cor} For integers $k,r\geq 0$,
\begin{align}\label{Hypergenfunvtion0-main-resut-import-cor}
\sum_{n=1}^\infty \frac{\binom{2n}{n}}{n 4^n}\zeta_n^\star(\{2\}_r)t_n(\{1\}_k)=2\sum_{i+j=k,\atop i,j\geq 0} \sum_{n=1}^\infty  \frac{(-1)^{n-1}}{n^{2r+1}} t_n(\{1\}_i)t_n^\star(\{1\}_j),
\end{align}
where the \emph{multiple $t$-harmonic sums} and \emph{multiple $t$-harmonic star sums} are defined by
\begin{align*}
t_n(\bfk):=\sum_{n\geq n_1>\cdots>n_r>0} \prod_{j=1}^r \frac{1}{(n_{j}-1/2)^{k_{j}}}\quad
\text{and}\quad
t^\star_n(\bfk):=\sum_{n\geq n_1\geq \cdots\geq n_r>0} \prod_{j=1}^r \frac{1}{(n_{j}-1/2)^{k_{j}}}.
\end{align*}
If $n<r$ then ${t_n}(\bfk):=0$ and let ${t_n}(\emptyset)={t^\star_n}(\emptyset ):=1$. When taking the limit $n\rightarrow \infty$ we get the so-called \emph{multiple $t$-values} and \emph{multiple $t$-star values}, respectively, see \cite{H2019}. In fact, Zhao \cite{Zhao2015} had begun studying some sum formulas for multiple $t$-values a few years prior to Hoffman's formal definition of multiple $t$-values.
\end{cor}

{\bf Declaration of competing interest.}
The author declares that he has no known competing financial interests or personal relationships that could have
appeared to influence the work reported in this paper.

{\bf Data availability.}
No data was used for the research described in the article.

{\bf Acknowledgments.}  The author is grateful to Dr. Steven Charlton and Kamcheong Au for their valuable insights on the evaluation of the identity in \eqref{thm-one-equ-hydigama}, and to Prof. Wenchang Chu for assessing \eqref{Hypergenfunvtion0-main-resut2}. The author is supported by the General Program of Natural Science Foundation of Anhui Province (Grant No. 2508085MA014).


\medskip

\begin{thebibliography}{99}

\bibitem{A2000}
G.E.\ Andrews, R.\ Askey and R.\ Roy, \emph{Special functions},
Cambridge University Press, 2000, pp.\ 66, 130.

\bibitem{Apery1978}
R.\ Ap\'ery, Irrationalit\'e de $\zeta(2)$ et $\zeta(3)$ (in French),
\emph{Ast\'erisque} \textbf{61} (1979), pp.\ 11--13.

\bibitem{AM1999}
T. Arakawa and M. Kaneko, Multiple zeta values, poly-Bernoulli numbers, and related zeta functions, \emph{Nagoya Math. J.} {\bf 153}(1999), 189-209.

\bibitem{Au2024}
K. Au, Multiple zeta values, WZ-pairs and infinite sums computations, \emph{Ramanujan J.} \textbf{66}(3)(2024).

\bibitem{CantariniD2019}
M.\ Cantarini and J.\ D'Aurizio,
On the interplay between hypergeometric series, Fourier--Legendre expansions and Euler sums,
\emph{Boll.\ Unione Mat.\ Ital.}, \textbf{12} (2019), pp.\ 623--656.

\bibitem{Chen2016}
H. Chen, Interesting series associated with central binomial coefficients, Catalan numbers and harmonic numbers, \emph{J.
Integer Seq.} \textbf{19}(1)(2016) Article 16.1.5, 11 pp.

\bibitem{ChenWangZhong2025}
A. Chen, W. Wang and J. Zhong, Evaluations of some double Ap\'ery-type series via Fourier-Legendre expansions, \emph{Bull. Malays. Math. Sci. Soc.} \textbf{48}(4)(2025), Paper No. 115, 25 pp.

\bibitem{ChenWang2025}
X. Chen and W. Wang, Ap\'ery-type series via colored multiple zeta values and Fourier-Legendre series expansions, \emph{J. Symbolic Comput.} (2025) 134:102508.

\bibitem{GR2025}
M. Gen\v{c}ev and P. Rucki, On a class of multiple Ap\'ery-like series and their reduction, \emph{Mediterr. J. Math.}, (2025) 22:195.

\bibitem{GR}
I.S. Gradshteyn and I.M. Ryzhik, Table of Integrals, Series, and Products, 8th ed., Academic Press, 2015.

\bibitem{H1992}
M.E. Hoffman, Multiple harmonic series, \emph{Pacific J.\ Math.} \textbf{152}(1992), pp.\ 275--290.

\bibitem{H2019}
M.E. Hoffman, An odd variant of multiple zeta values, \emph{Comm. Number Theory Phys.} \textbf{13}(2019), pp.\ 529--567.

\bibitem{KanekoWangXuZhao2022}
M. Kaneko, W. Wang, C. Xu and J. Zhao, Parametric Ap\'{e}ry-type series and Hurwitz-type multiple zeta values, arXiv:2209.06770.

\bibitem{Ku2010}
M. Kuba, On functions of Arakawa and Kaneko and multiple zeta values, \emph{Appl. Anal. Discrete Math.} {\bf 4}(2010), pp.\ 45-53.

\bibitem{LaiLuorr2022}
L. Lai, C. Lupu and D. Orr, Elementary proofs of Zagier's formula for multiple zeta values and its odd variant, to appear in \emph{Proc. Amer. Math. Soc.} 	arXiv:2201.09262.

\bibitem{LuoSi2023}
F. Luo and X. Si, A note on Arakawa-Kaneko zeta values and Kaneko-Tsumura $\eta$-values, \emph{Bull. Malays. Math. Sci. Soc.} \textbf{46}(2023), 21.

\bibitem{LiYan2025}
Z. Li and L. Yan, Generating functions of multiple $t$-star values of general level, \emph{Adv. Appl. Math.} {\bf 165}(2025), 102853.

\bibitem{Lupu2022}
C. Lupu, Another look at Zagier's formula for multiple zeta values involving Hoffman elements, \emph{Math. Zeit.} \textbf{301}(2022), pp.\ 3127-3140.

\bibitem{XuZhao2021c}
C. Xu and J. Zhao, Ap\'{e}ry-type series and colored multiple zeta values, \emph{Adv. Appl. Math.} {\bf 153}(2024), 102610.

\bibitem{XuZhao2020d}
C.\ Xu and J. Zhao, Explicit relations of some variants of convoluted multiple zeta values, \emph{Ann. Math. Pura Appl.} (2025), https://doi.org/10.1007/s10231-025-01561-4.

\bibitem{DZ1994}
D. Zagier, Values of zeta functions and their applications, First European Congress
of Mathematics, Volume II, Birkhauser, Boston, \textbf{120}(1994), pp.\ 497--512.

\bibitem{Zhao2015}
J. Zhao, Sum formula of multiple Hurwitz-zeta values, \emph{Forum Math.} \textbf{27}(2)(2015), pp.\  929--936.

\bibitem{Z2016}
J. Zhao, \emph{Multiple zeta functions, multiple polylogarithms and their special values}, Series on Number
Theory and its Applications, Vol.~12, World Scientific Publishing Co. Pte. Ltd., Hackensack, NJ, 2016.

\bibitem{ZY2023}
W. Zheng and Y. Yang, Parametric Arakawa-Kaneko zeta function and Kaneko-Tsumura $\eta$-function, \emph{Bull. Belg. Math. Soc. Simon Stevin}, \textbf{30}(2023), pp.\ 328--340.
\end{thebibliography}
\end{document}